\documentclass[12pt]{article}
\usepackage[dvips]{graphicx}   
\usepackage{times}     
\usepackage[T1]{fontenc}  
\usepackage{graphicx,enumerate,amsmath,amsthm,mathrsfs,array,pstricks,}
\usepackage{amsfonts}
\usepackage{amssymb}
\usepackage[a4paper,left=2.5cm,right=3cm,top=2.5cm,bottom=2.5cm]{geometry}

\theoremstyle{}
\newtheorem{thm}{Theorem}[section]
\newtheorem{cor}[thm]{Corollary}
\newtheorem{lem}[thm]{Lemma}

\theoremstyle{definition}

\theoremstyle{remark}

\begin{document}
\begin{titlepage}
\title{Some Properties of Proper Power Graphs in Finite Abelian Groups}
\author{Dhawlath. G$^{1}$ and Raja. V$^{2}$}
\date{{\footnotesize Department of Mathematics, School of Advanced Sciences, VIT-AP University, Vijayawada, Andhra Pradesh, India}\\
{\footnotesize $^{1}$: dhawlath.21phd7154@vitap.ac.in\  $^{2}$: raja@vitap.ac.in }}
\maketitle
\renewcommand{\baselinestretch}{1.2}\normalsize

\begin{abstract}
The power graph of a group $G$, denoted as $P(G)$, constitutes a simple undirected graph characterized by its vertex set $G$. Specifically, vertices $a,b$ exhibit adjacency exclusively if $a$ belongs to the cyclic subgroup generated by $b$ or vice versa. The corresponding proper power graph of $G$ is obtained by taking $P(G)$ and removing a vertex corresponding to the identity element, which is denoted as $P^*(G)$. In the context of finite abelian groups, this article establishes the sufficient and necessary conditions for the proper power graph's connectedness. Moreover, a precise upper bound for the diameter of $P^*(G)$ in finite abelian groups is provided with sharpness. This article also explores the study of vertex connectivity, center, and planarity. 
\end{abstract}

\section{Introduction}

\par In 2002, Kelarev and Quinn \cite{kelarev2002directed} made a significant contribution to the field by introducing a concept known as the directed power graph for semigroups. This idea was further developed by Chakrabarty \cite{chakrabarty2009undirected}, who introduced the undirected power graph for group $G$. The power graph, denoted as $P$($G$), is a graph where each element in the group $G$ corresponds to a vertex, and two vertices $a$ and $b$ are connected by an edge if and only if either $a$$\in$$\langle$$b$$\rangle$ or $b$$\in$$\langle$$a$$\rangle$. Chakrabarty has studied these graphs in depth in order to comprehend their features and attributes in different situations. Notably, they have established the circumstances in which these graphs are regarded as Hamiltonian, complete, or planar. Moreover, these structures have gained a mathematical layer of insight from T. Tamizh Chelvam and M. Sattanathan \cite{chelvam2018power}, who have determined the circumstances under which $P$($G$) becomes Eulerian. S. Chattopadhyay \cite{chattopadhyay2018connectivity} investigated the ideas of planarity and vertex connectivity inside finite cyclic groups in a different work. Furthermore, a fascinating connection between finite abelian groups has been revealed by the work of P.J. Cameron and Gosh \cite{cameron2011power} they proved that two such groups are isomorphic if and only if the corresponding power graphs are isomorphic as well, providing a significant connection between the graphical representation these power graphs provide and the algebraic structure of groups.\\
\par In \cite{chakrabarty2009undirected} The power graph $P$($G$) of a group $G$, is inherently connected. Notably, this connectivity is emphasized by the observation that in the power graph, every vertex is adjacent to the vertex corresponding to the identity element. An interesting facet of this study involves the creation of $P^\ast(G)$, by removing the vertex associated with the identity element from $P$($G$). In \cite{aalipour2016structure} Alipour posed an open problem, raising the question of which groups exhibit the property that their power graph remains connected even when the identity element is removed. Subsequently, numerous researchers studied the investigation of proper power graphs. In \cite{curtin2015punctured} Curtin determines the specific kinds of groups for which the diameter of $P$$^\ast$($G$) is at most 2. Furthermore, Curtin identified groups for which the diameter of the proper power graph is precisely 3 and Cutrin demonstrates that $P$$^\ast$($G$) is Eulerian if and only if $G$ belongs to a specific category, namely, a cyclic 2-group or a generalized quaternion 2-group \cite{curtin2015punctured}. Doostabadi extended the exploration of proper power graphs, showcasing the connectivity of such graphs for certain finite groups. An important finding was the establishment that \cite{doostabadi2015connectivity} the proper power graph attains a diameter of at most 26 whenever it is connected.\\
\par Consider a group $G$, and let $a$ and $b$ be elements in $G$. If the subgroup generated by $a$, denoted as $\langle$$a$$\rangle$, is a cyclic group, and if $\langle a \rangle$ spans the entire group $G$ then $G$  is cyclic and in this paper, $C_n$ denoted as a cyclic group of order $n$. Here, $O(a)$ represents the order of the element $a$, and $O(a)= \mid \langle a\rangle \mid$. In the context of finite groups, a $p$-group of exponent $p$ is a group where every non-identity element has order $p$. Let $G$ be a graph, and let us now review some fundamental definitions of graph theory that we have used in this paper. A graph $G$ is connected if each pair of vertices in $G$ belongs to a path otherwise, $G$ is disconnected \cite{west2001introduction}. In this paper, the adjacency of two vertices, denoted as $a \sim b$, is established. The vertex connectivity of $G$, denoted as $\kappa(G)$, is the minimum size of a vertex set $S$ such that $G-S$ is disconnected or has only one vertex \cite{west2001introduction}. Additionally, For a graph $G$, represented as $diam(G)$, it is defined as the maximum length of the shortest path between any two vertices in the graph. In this context, the eccentricity of a vertex is the greatest distance from that vertex to any other vertex within the graph. The center of graph $G$ is a subgraph induced by the vertices with the minimum eccentricity \cite{west2001introduction}. A graph $G$ is considered planar if it can be positioned on a two-dimensional plane in such a way that none of its edges cross \cite{west2001introduction}. Kuratowski's theorem is a crucial result in the study of planar graphs. According to Kuratowski's theorem \cite{west2001introduction}, a graph is planar if and only if it does not contain a subgraph that is a subdivision of the complete graph $K_5$(the complete graph on five vertices) or the complete bipartite graph $K_{3,3}$ (a bipartite graph with three vertices in each partition and all possible edges). This theorem offers a practical and straightforward method to determine if it is possible to represent a given graph on a plane without any instances of edges crossing.

\section{Foundational Lemmas}

This section includes essential lemmas that form the foundation for proving our main results.

\begin{lem}\label{DiviReg}
Let $G$ be finite group $a,b \in  G$ have same order then either $\langle a\rangle$=$\langle b\rangle$ or $\langle a\rangle\cap\langle b\rangle= e$.
\end{lem}

\begin{proof}
Let $a,b \in G$ and the order of $a$ and $b$ be equal. We know that $\langle a\rangle=\{a,a^2,a^3,...,a^n\}$,
  $\langle b\rangle$=$\{b,b^2,b^3,...,b^n\}$.
  Suppose, $b_i\in\langle a\rangle, b=a^k$ for some integer $k$, and 
  ($a^k)^m=a^{km}$,
  $b^m=a^{km}$. Since both $k$ and $m$ are integers, their product $km$ is also an integer, therefore, any power of $b$ can be expressed as a power of $a$. Since $a$ and $b$ have the same order,  then $\langle a\rangle=\langle b\rangle$. Suppose $b_i\notin\langle a\rangle$ and  $a_i\notin\langle b\rangle$
   since it is a cyclic subgroup it has an identity in common, therefore, $\langle a\rangle\cap\langle b\rangle= e$.
\end{proof}

\begin{lem}\label{DiviRegDom}
In $P(G)$, two vertices $a$ and $b$ are adjacent then either $O(a)\mid O(b)$ or $O(b)\mid O(a)$.
\end{lem}
\begin{proof}
    Let $a,b\in P(G)$, if two vertices $a$ and $b$ are adjacent in $P(G)$, then either $a \in\langle b\rangle$ or $b\in\langle a \rangle$. If $a\in\langle b \rangle$ then $\langle a\rangle\subseteq\langle b\rangle$. Also Similar for $b$, since $\langle a\rangle$   and  
$\langle b\rangle$ are cyclic subgroups, from this we can say that either $O(a)\mid O(b)$ or $O(b)\mid O(a)$. 
\end{proof}
\begin{lem}\label{domcor}
Let $G$ be a finite cyclic group of order $n$ and $a,b\in G$. In $P(G$),  two vertices $a$ and $b$ are adjacent if and only if either $O(a)\mid O(b)$ or $O(b)\mid O(a)$.
\end{lem}
\begin{proof}
     Here, the necessary condition for this lemma is direct from lemma \ref{DiviRegDom}. The sufficient condition for this lemma by $P$($G$) definition is, that two vertices $a$ and $b$ are adjacent if and only if either $a\in\langle b\rangle$ or $b\in\langle a\rangle$. Now it is enough to show that if  $O(a)\mid O(b)$ then $a\in\langle b\rangle$. Since $G$ is a finite cyclic group,  there exists  $g\in G$ such that $\langle g\rangle=G$ and if $k$ is a positive divisor of $n$ in a cyclic group of order $n$, then the number of elements of order $k$ is $\phi(k)$ \cite{gallian2021contemporary}. Let us take $O(a) =m$. In $G$, there exist $\phi(m)$ elements with an order of $m$, let us take $l_1,l_2,....,l_m$ are the elements of order $m$ and all these elements belong to  $\langle g\rangle$, the elements that generate a cyclic subgroup are also contained within $\langle g\rangle$. We are aware that both $\langle a\rangle$,$\langle b\rangle\subseteq\langle g\rangle$ and $\langle b\rangle$ is a cyclic subgroup. So from given condition  $O(a)\mid O(b)$, then $m$ divides $O(b)$, therefore, in $\langle b\rangle$ there exist $\phi(m)$ elements that have order $m$, and since $\langle b\rangle$ is a subgroup of $G$ and contains all cyclic subgroups generated by $m$ order elements, the $m$ order elements in $G$ and $\langle b\rangle$ are the same in this case. Element $a$ is also included among all the elements with order $m$. Therefore, $a\in\langle b\rangle$. Similarly if $O(b)\mid O(a)$ then $b\in\langle a\rangle$.
\end{proof}
\section{Connectivity}

This section includes the necessary and sufficient conditions for finite abelian groups whose proper power graph exhibits connectivity. Additionally, we establish a sharp upper bound on the diameter of $P^\ast(G)$ in finite abelian groups. Furthermore, the section covers discussions on vertex connectivity and the center of $P^\ast(G)$ for finite abelian groups.
\begin{cor}
    Let $G$ be a finite $p$-group, where $p$ is a prime and then $P^\ast(G)$ is connected if and only if $G$ is either cyclic or generalized quaternion \cite{moghaddamfar2014certain}.
\end{cor}

\begin{thm}\label{cccartub}
Let $G$ be a finite group such that Z($G$) is not a $p$-group and then $P^\ast(G)$
is connected and Moreover, diam($P^\ast(G)$) $\leq$6 and the bound is sharp \cite{doostabadi2015connectivity}.
\end{thm}
An immediate outcome of Theorem \ref{cccartub} is as follows:

\begin{cor}\label{cccartGKn}
Let $G$ be a finite abelian group of order n which is not a $p$-group, then $P$$^\ast$($G$) is connected.
\end{cor}
\begin{lem}\label{da}
    Let $G$ be a  finite non-cyclic abelian group of order $p^n,n\geq$2, then $P^\ast(G)$ is disconnected.
\end{lem}
\begin{proof}
    Since the order of the group $G$ is $p^n,n\geq$2, this is the $p$-group. By our assumption $G$ is non-cyclic abelian, it is disconnected according to \cite{moghaddamfar2014certain},  $P^\ast(G)$ should either be cyclic or a generalized quaternion.
\end{proof}
\begin{thm}\label{con}
    For a finite abelian group $G$ with an order of $n$. The proper power graph $P^\ast(G)$ is connected if and only if $G$ is not a non-cyclic group of order $p^n$, where $n \geq 2$.
\end{thm}
\begin{proof}
     Since $G$ is a finite abelian group, every cyclic subgroup is abelian. Thus, it can be separated into a cyclic abelian group and a non-cyclic abelian group.
    In the cyclic group, the vertices corresponding to generator elements are adjacent to every vertex in the power graph since we remove the identity and still it is connected. So, for every cyclic group of proper power graph is connected. Now we consider the non-cyclic finite abelian group.\\
    The necessary condition for this theorem is direct from the lemma \ref{da}.\\
    The sufficient condition for this theorem is if $G$ is not a non-cyclic group of order $p^n,n\geq$2, then $P^\ast(G)$ is connected.
     If $G$ is a group and the order of $G$ is some prime $p$, then $G$ is cyclic, therefore,  $P$$^\ast$($G$) is connected. And if $G$ is a group of order $n$ which is not a $p$-group, then by corollary \ref{cccartGKn}  $P$$^\ast$($G$) is connected. Here, we have demonstrated that excluding non-cyclic groups of order $p^n$ where $n \geq 2$, the proper power graph $P^\ast(G)$ remains connected.  
\end{proof}
\begin{lem}\label{cccart1cycle}
 Let $G$ be a finite group and $a,b\in G\backslash\{e\}$ such that $ab=ba$ and \\ $gcd(\mid a\mid,\mid b\mid)=1$ then $a\sim ab\sim b$ \cite{doostabadi2015connectivity}.
\end{lem}
\begin{thm}\label{leq}
     Let $G$ be a finite abelian group and $P^\ast(G)$ is connected then diam($P^\ast(G))$ $\leq$ 4.
\end{thm}
\begin{proof}
 In the context of a finite abelian group $G$, we aim to establish this proof in two cases: one where $G$ is a $p$-group and the other where it is not.\\
   \textbf{case (i)}: By \cite{doostabadi2015connectivity}, if $P^\ast(G)$ is connected then $G$ is either cyclic or generalized quaternion. Every cyclic group is known to be abelian, and the generalized quaternion is not. So in the cyclic group of $P^\ast(G)$, the vertex corresponding to generator elements is adjacent to every vertex, therefore, $diam(P^\ast(G)) \leq 2$.\\
    \textbf{case (ii)}: Theorem \ref{con} tells us that $P^\ast(G)$ is connected if $G$ is not a $p$-group. Consider two distinct primes, $p$ and $q$, and let $x_p, x_q \in G$ be the elements of order $p$ and $q$, respectively. Assume that $a, b$ are arbitrary elements in $G$ with $\mid a \mid = n$ and $\mid b \mid = m$. Since $G$ is abelian if $\gcd(n,m) = 1$, then by lemma \ref{cccart1cycle}, $a \sim ab \sim b$, therefore, $d(a,b) \leq 2$. If $\gcd(n,m) \ne 1$, then we choose $s \in \{p,q\}\backslash\{r\}$, in which $r$ is a prime factor that both $m$ and $n$ share. Now we have $\gcd(n,\mid x_s \mid) = 1$, so $a \sim ax_s \sim x_s$ and $\gcd(\mid x_s\mid,m) = 1$, so $x_s \sim x_sb \sim b$ from this $a\sim ax_s \sim x_s \sim x_sb \sim b$, so $d(a,b) \leq 4$. 
\end{proof}

In the paper \cite{chakrabarty2009undirected}, the authors assert that $P(G)$ is always connected and they establish that if $G$ is a cyclic group of order 1 or $p^n$, for some prime number $p$ and for some $n\in N$, where $N$ is a natural number then $P(G)$ is complete.\\
We understand that the removal of the identity vertex from $P(G)$ results in the deletion of one vertex from the complete graph $K_n$, thereby producing a complete graph $K_{n-1}$. This leads to the conclusion that, for a cyclic group $G$, the elements generating $G$ are adjacent to all other vertices. Put more plainly, to form a complete graph on $K_{n-1}$, $P^\ast(G)$ must produce a finite cyclic group of order $n$, which is classified as a $p$-group. 
Conversely, $P^\ast(G)$ does not constitute a complete graph if $G$ is not a $p$-group.
\begin{thm}\label{cccartGH}
 For a finite cyclic group $G$ with an order of $n$, where $\kappa(P^\ast(G))$ represents the vertex connectivity of the proper power graph of $G$, satisfying the following condition.
       \begin{enumerate}
           \item[(i)]  $\kappa$($P^\ast(G)$) = $n-2$, when $G$ is a p-group.
           \item[(ii)]  $\kappa$($P^\ast(G)$) = $\phi$(n), when the order of the group $G$ is $pq$, where $p$ and $q$ are distinct primes.
           \item [(iii)] $\kappa$($P^\ast(G)$) $>$ $\phi$(n), when $G$ is neither a $p$-group nor a cyclic group of order $pq$.
       \end{enumerate}
\end{thm}

\begin{proof}
      \item [(i)] Given that $G$ is a cyclic $p$-group, the proper power graph $P^\ast(G)$ forms a complete graph of order $K_{n-1}$. Therefore  $\kappa$($P^\ast(G)$) = $n-2$.
      \item [(ii)] Since $G$ is a finite cyclic group, so in $P^\ast(G)$ there are $\phi(n)$ vertices adjacent to every vertices. By lemma \ref{DiviRegDom}, if two vertices are adjacent in $P^\ast(G)$ then the order of one element should divide the other, here the order of $G$ is $pq$, therefore, there exist at least two distinct prime order elements in $G$ that are not adjacent, now we need to prove that these two elements do not have common neighbor apart from generators of $G$. The possible adjacent vertices of order $p$ elements are $p^n$ where $n$ is a natural number and $pq$ similarly for $q$ order elements is $q^n$ and $pq$ if the elements have order $pq$ then it is the generator of $G$, from this, the common neighbor for $p$ and $q$ are generators of $G$ only. So removing the generator elements form the $P^\ast(G)$ gives disconnected graph, therefore,  $\kappa$($P^\ast(G)$) = $\phi(n)$.
      \item [(iii)] Here, in this case, $\phi(n)$ vertices are adjacent to all vertices in  $P^\ast(G)$. Now we are going to prove that after removing $\phi(n)$ vertices from $P^\ast(G)$ still the graph is connected. Since it is not a $p$-group and the order is not $pq$, given any two elements in this group that do not divide each other's order but both divide a common element order apart from the generating elements.  By lemma \ref{domcor}, $a,b\in$ $P^\ast(G)$, $a$ and $b$ are adjacent if order of $a$ divides order of $b$. So, from this after removing $\phi(n)$ vertices from $P^\ast(G)$, given any two non-adjacent vertex, a common neighbor exists which is not correspond to the generator elements. From this we can say that after removing $\phi(n)$ vertices still the graph is connected, therefore, $\kappa$($P^\ast(G)$) $>$ $\phi(n)$.        
       
\end{proof}
\begin{lem}\label{Na}
    If $G$ is a finite non-cyclic abelian group of order $n$ which is not a $p$-group, then it has at least two vertices corresponding to the same prime order element in $P^\ast(G)$ that are not adjacent.
\end{lem}
\begin{proof}
    Since $G$ is a finite non-cyclic Abelian group, By the Fundamental Theorem of Finite Abelian Groups, each finite abelian group $G$ can be represented as being isomorphic to a group in the form of $C_{p1}^{n1} \oplus C_{p2}^{n2} \oplus...\oplus C_{pk}^{nk}$, where the prime power order cyclic group is denoted by $C_{pi}^{ni}$. Since it is non-cyclic, at least two $p_i$ must be equal, otherwise it will result in a cyclic group and here elements are in tuple form we select two elements that have the same order but are not adjacent. Let us take $q$ is the same prime and $C_q^{n1} \oplus C_{p2}^{n2} \oplus C_q^{n2} \oplus...\oplus C_{pk}^{nk}$ is isomorphic to a group $G$, and let ($a$,0,0,...,0), (0,0,$a$,...0) be the two elements in $G$, where $a$ is the prime order element in $C_q^n$. ($a$,0,0,...,0)$\notin$$\langle$(0,0,$a$,...0)$\rangle$ and (0,0,$a$,...,0)$\notin$$\langle$($a$,0,0,...0)$\rangle$. Therefore, here the two vertices corresponding to these elements whose order is the same but not adjacent in $P^\ast(G)$.
\end{proof}
\begin{thm}
     Let $G$ be a finite non-cyclic abelian group which is not a $p$-group, then $\kappa$($P^\ast(G)$) $\leq$ $\mid S\mid -\mid r\mid$, where $S$ is the set of all prime order element of $G$. 
\end{thm}
\begin{proof}
    By lemma \ref{Na}, there exist two vertices corresponding to the same prime order element  $a,b\in P^\ast(G)$  which are not adjacent, and let us take two elements $a$ and $b$ whose order is $p$. By lemma \ref{cccart1cycle}, $c,d\in G\backslash\{e\}$ such that $cd=dc$ and $\gcd(\mid c\mid,\mid d\mid)=1$, then $c\sim cd\sim d$. All of the elements in this group commute with one another because it is an abelian group. Here $\gcd(\mid a\mid,\mid b\mid)\ne 1$, so we can give a path between these vertices $a$ and $b$. Since it is not a $p$-group, it has different prime order elements, let's take $q_i$ as another prime order element, now $gcd(\mid a\mid,\mid q_i\mid)=1$ then $a\sim aq_i\sim q_i$ and $\gcd(\mid q_i \mid,\mid b\mid)=1$ then $q_i \sim q_ib \sim b$. From this, we write $a \sim aq_i\sim q_i\sim q_ib\sim b$. If two vertices corresponding to the same prime order element are not adjacent, then it has to be connected only through a vertex that corresponds to another prime order element. So we can disconnect the graph if we remove all vertex corresponding to the remaining prime order element. As we assume $S$ is the set consisting all the prime order elements of group $G$, $r$ is equal to $S\backslash\{p_i\}$, the set of all $p$ order elements is denoted by $p_i$ and $r$$\subset$$S$,  from this $\kappa(P^\ast(G)) \leq \mid S \mid-\mid r\mid$.
\end{proof}
\begin{thm}\label{cen}
    Let $G$ be a finite cyclic group of order $n$. If $P^\ast(G)$ is connected, then the center of $P^\ast(G)$ forms a complete subgraph.
\end{thm}
\begin{proof}
    In the context of a finite cyclic group $G$, we aim to establish this proof in two cases: one where $G$ is a cyclic  $p$-group and the other where it is not.\\
    \textbf{case(i)} If $G$ is a finite cyclic $p$-group, then $P^\ast(G)$ forms a complete graph, resulting in the same eccentricity for every vertex. Consequently, in this case, $P^\ast(G)$ serves as the center of $P^\ast(G)$.\\
    \textbf{case(ii)} Consider $G$ be a finite cyclic group of order $n$, which is not a $p$-group. Here $P^\ast(G)$ is connected and it is a cyclic group, since all vertex in $P^\ast(G)$ are adjacent to the vertices corresponding to the generator element, these vertices have minimum eccentricity and now we have to show that only vertices corresponding to the generator element in $P^\ast(G)$ are adjacent to every vertex. Suppose we take vertex $c\in$ $P^\ast(G)$ which corresponds to the non-generator element that is adjacent to every vertex in $P^\ast(G)$ and $a_i$ is some random vertices in $P^\ast(G)$ then by lemma \ref{domcor}, either $O(c) \mid O(a_i)$ or $O(a_i) \mid O(c)$. As it is not a $p$-group, there exist at least two vertices in $P^\ast(G)$ that correspond to distinct prime order elements, since this vertex does not correspond to a generator element, so, $O(c) \nmid O(a_i)$ and $O(a_i) \nmid O(c)$, therefore, it is a contradiction. In this case, only the vertices corresponding to generator elements are adjacent to every other vertex, and in a cyclic group, there are $\phi(n)$ generators, since the distance of these vertices corresponding to the generator is 1, the center of $P^\ast(G)$ is a complete graph of order $\phi(n)$.
\end{proof}
\begin{thm}
     For a finite non-cyclic abelian group $G$ with an order of $n$. Let $Z$ be a cyclic subgroup generated by $q_1,q_2\ldots ,q_n$, where $q_1, q_2, \ldots, q_n$ are distinct primes in $G$. If $P^\ast(G)$ is connected, then the center of $P^\ast(G)$ is $P^\ast(Z)$.
\end{thm}
\begin{proof}
     Consider $G$ to be a finite non-cyclic abelian group of order $n$. By theorem \ref{con}, if $G$ is a non-cyclic group of order $p^n$, then $P^\ast(G)$ is disconnected. If the group has an order of $p$, where $p$ is a prime number, it is inherently a cyclic group. Hence, the focus shifts to establishing that $G$ is a finite non-cyclic abelian group with an order of $n$ that does not qualify as a $p$-group. By theorem \ref{leq} $diam(P^\ast(G)) \leq 4$. According to the Fundamental Theorem of Finite Abelian Groups, every finite abelian group $G$ can be represented as being isomorphic to a group in the form of $C_{p1}^{n1} \oplus C_{p2}^{n2} \oplus...\oplus C_{pk}^{nk}$, where $C_{pi}^{ni}$ are the cyclic group of prime power order. Here at least two $p_i$ must be equal, otherwise, it will result in a cyclic group. Within the direct product of a prime power ordered cyclic group, some distinct prime order will be there, let us take the prime order elements as $q_1, q_2, \ldots, q_n$, and we know that these are distinct primes, so in $G$, $q_1-1$ elements of order $q_1$, $q_2-1$ elements of order $q_2$,.., $q_n-1$ elements of order $q_n$ will be there.
     \par Let us take $a_i$ as some random elements in $G$, and $b_i$ denote an element of $G$ considering all possible combinations of $q_1, q_2, \ldots, q_n$. If $gcd(\mid a_i \mid,\mid b_i \mid)=1$, then by lemma \ref{cccart1cycle}, $a_i$$\sim$$a_ib_i$$\sim$$b_i$ and $gcd(\mid a_i\mid,\mid b_i \mid) \neq 1$ then it has some common prime divisor let us take $q_i$, so $q_i-1$ elements of order $q_i$ is contained in both the cyclic subgroup generated by $a_i$ and $b_i$ respectively, so $d(a_i,b_i) \leq 2$. Now, we have to show that only the vertex corresponding to these elements has eccentricity 2. Within the direct product of a prime power ordered cyclic group, we know that some common prime is there or else it will give a cyclic group. Let us name the common prime as $p_i$ and let us take $c_i$ as the element of $G$ which has prime divisor $p_i$. By lemma \ref{Na}, there exist at least two vertices corresponding to the same prime order element in $P^\ast(G)$ that are not adjacent. The distance between these two vertices corresponding to the prime order element is 4. Every $c_i$ has a prime divisor $p_i$. If we consider any vertex $c_i$ in $P^\ast(G)$, we know that $p_i$ is a common prime. Therefore, there are at least two vertices with the same prime order elements that are not adjacent to each other. From these two vertices, $c_i$, is adjacent to only one vertex and the distance between $c_i$ and another vertex corresponds to prime order element $p_i$, by theorem 3 $c_i\sim c_iq_i\sim q_i\sim q_ip_i\sim p_i$, $d(c_i,p_i)=4$ and if $c_i$ have $q_i$ order element then $c_i \sim q_i\sim q_ip_i\sim p_i$, so here $d(c_i,p_i)=3$, therefore $d(c_i,p_i) \geq 3$, so the vertices correspond to remaining element have eccentricity at least 3.
     \par Now we collect all element vertices which have eccentricity 2, thus, every possible combination of  $q_1,q_2,\ldots,q_n$ can be found in $Z$, the cyclic subgroup formed by $q_1,q_2,\ldots, q_n$. We know that only these vertices are known to have eccentricity 2. As a result, $P$$^\ast$($Z$) is the centre of $P$$^\ast$($G$).
\end{proof}
\section{Planarity}
\begin{lem}\label{CS}
    If $G$ is a finite group of order n and $p$ is the prime that divides the order of $G$ then $P($G$)$ has a K$_p$ subgraph.
\end{lem}
\begin{proof}
    By Cauchy theorem, $G$ has an element of order $p$   \cite{gallian2021contemporary}. Hence $\langle p\rangle=\{p^1,p^2,..,p^p\}$, and the cyclic subgroup generated by all the elements except identity in that cyclic subgroup generated by $p$ are equal, so, all vertex corresponding to these elements are pairwise adjacent in $P(G)$, since this group is finite, every cyclic subgroup generated by each element in $G$ has the identity element as a member, therefore, every vertex are adjacent to the vertex corresponding to the identity element, this implies that $P(G)$ has $K_p$ subgraph.
\end{proof}
\begin{thm}
     Consider a finite cyclic group $G$ of order $n$. The proper power graph $P^\ast(G)$ is non-planar if and only if $n \geq 7$.
\end{thm}
\begin{proof}
    We know that $G$ has $\phi(n)$ generators because it is a cyclic group of order $n$. Additionally, all vertex corresponding to these elements are adjacent to every vertex in $P^\ast(G)$. If $n\geq$7, then $\phi(n)\geq4$, therefore, if the order of the group is greater than 7, then in $P^\ast(G)$ has at least 4 generator element vertices. So in $P^\ast(G)$, we take three vertices corresponding to the generator element and any three random vertices, all vertex corresponding to the generator element are adjacent to every vertex, and the random vertices are adjacent to the vertices corresponding to the generator element. Therefore, it has $K_{3,3}$ by Kuratowski’s theorem  $P^\ast(G)$ is non-planar.
  \par The converse aspect of the theorem serves as conclusive proof for the planarity of $P^\ast(G)$ when $n \leq 6$. Specifically, if $n \leq 5$, then $P^\ast(G)$ forms a complete graph $K_{n-1}$. If $n=6$ then $P^\ast(G)$ will give a five vertices simple graph which is not a complete graph. So here, $P^\ast(G)$ does not contain $K_5$ and $K_{3,3}$. Therefore by Kuratowski’s theorem, it is planar.
\end{proof}
\begin{thm}
    Let $G$ be a finite $p$-group of exponents $p$ then $P^\ast(G)$ is planar if and only if $p\leq$5.
\end{thm}
\begin{proof}
    Since $G$ is a finite $p$-group of exponents $p$ by \cite{doostabadi2015connectivity}, $P^\ast(G)$ is the union of the complete graph of order $p -1$. If $P^\ast(G)$ is planar then $p\leq 5$. Suppose, $p>5$ then the next prime is 7, so, $P^\ast(G)$ will be the union of $K_6$ and similarly $p>7$, then $P^\ast(G)$ will be the union of $K_{p-1}$, therefore, $P^\ast(G)$ has $K_5$ subgraph, hence, by Kuratowski’s theorem, it is established as non-planar. This poses a contradiction, leading to the conclusion that $P^\ast(G)$ is planar only when $p \leq 5$.
    \par Conversely when $p \leq 5$, $P^\ast(G)$ can be expressed as the union of complete graphs with orders less than 4. Applying Kuratowski’s theorem, we confirm the planarity of $P^\ast(G)$ since it does not contain both $K_5$ and $K_{3,3}$.
\end{proof}
\begin{thm}
     If $G$ is a finite non-cyclic abelian group of order $p^n$ where $p\geq 7$, then $P^\ast(G)$ is non-planar.
\end{thm}
\begin{proof}
    Here, $G$ is a non-cyclic abelian group of order $p^n$. By lemma \ref{CS} if any prime $p$ divides the order of the group then in $P^\ast(G)$ it has $K_{p-1}$ subgraph. Therefore, if $p=7$ then $P^\ast(G)$ has a $K_6$ subgraph, similarly, if $p>7$ then $P^\ast(G)$ has a $K_{P-1}$ subgraph, so in every graph $K_5$ subgraph will exist, therefore, it is non-planar by Kuratowski’s theorem.
\end{proof}
\begin{thm}
     Let $G$ be a finite group of the form  $C_q \oplus C_{p1} \oplus C_{p2} \oplus...\oplus C_{pk}$ where $p$ and $q$ are distinct prime and if $p$ and $q$ are 2 and 3 then $P$$^\ast$($G$) is planar.
\end{thm}
\begin{proof}
    Here, $G$ is a finite abelian non-cyclic group which is not a $p$-group so by theorem \ref{con}, $P^\ast(G)$ is connected. \\
  \textbf{case(i)} Here, we take $p$=2 and $q$=3, so the group $G$ is like $C_3 \oplus C_2 \oplus C_2 \oplus...\oplus C_2$, in this group,  two three-order elements and $p^k$-1 two-order elements and the remaining are six-order elements. In this graph the vertex corresponding to two-order elements are not adjacent, in $P^\ast(G)$ two vertices $a,b$ are adjacent then $a\in \langle b\rangle$ or $b \in \langle a\rangle$, cyclic subgroup generated by two-order elements have only two elements in that subgroup one is identity and another is that element itself. And every cyclic subgroup generated by six-order elements has one identity, one two-order element, two three-order elements, and two six-order elements, in that, one is its inverse element and another is itself and the cyclic subgroup generated by its inverse element is equal, so,  the vertex corresponding to three-order elements are adjacent to every vertex corresponding to six-order elements and every vertex corresponding to two-order element is adjacent two vertices corresponding to six-order element. Hence the degree of every vertex corresponding to a two-order element is two and the degree of every vertex corresponding to six-order element is four.

\par In $P^\ast(G)$ to show that the graph does not have $K_5$ and $K_{3,3}$ we choose the vertices corresponding to three-order and six-order elements, we take the two vertices corresponding to the three-order element, these vertices are adjacent to all vertex corresponding to six-order element and we have to choose three more vertices corresponding to six-order element which are pairwise adjacent to each other, we know that every vertex corresponding to six-order element is adjacent to only one vertex corresponding to the six-order element, so, here we can't find three vertices corresponding to the six-order element that are pairwise adjacent, so, $K_5$ is not possible, similarly, for $K_{3,3}$ we need at least one vertex corresponding to six-order element which is adjacent to three vertices corresponding to six-order elements, so, $K_{3,3}$ is not possible by Kuratowski’s theorem, $P$$^\ast$($G$) is planar. In the below graph, $q_i$ is the vertices corresponding to the three-order element, $p_i$ is the vertices corresponding to the two-order element and $r_i$ is the vertices corresponding to the six-order element.
\begin{figure}[htb!]
    \centering
    \includegraphics[width=100mm]{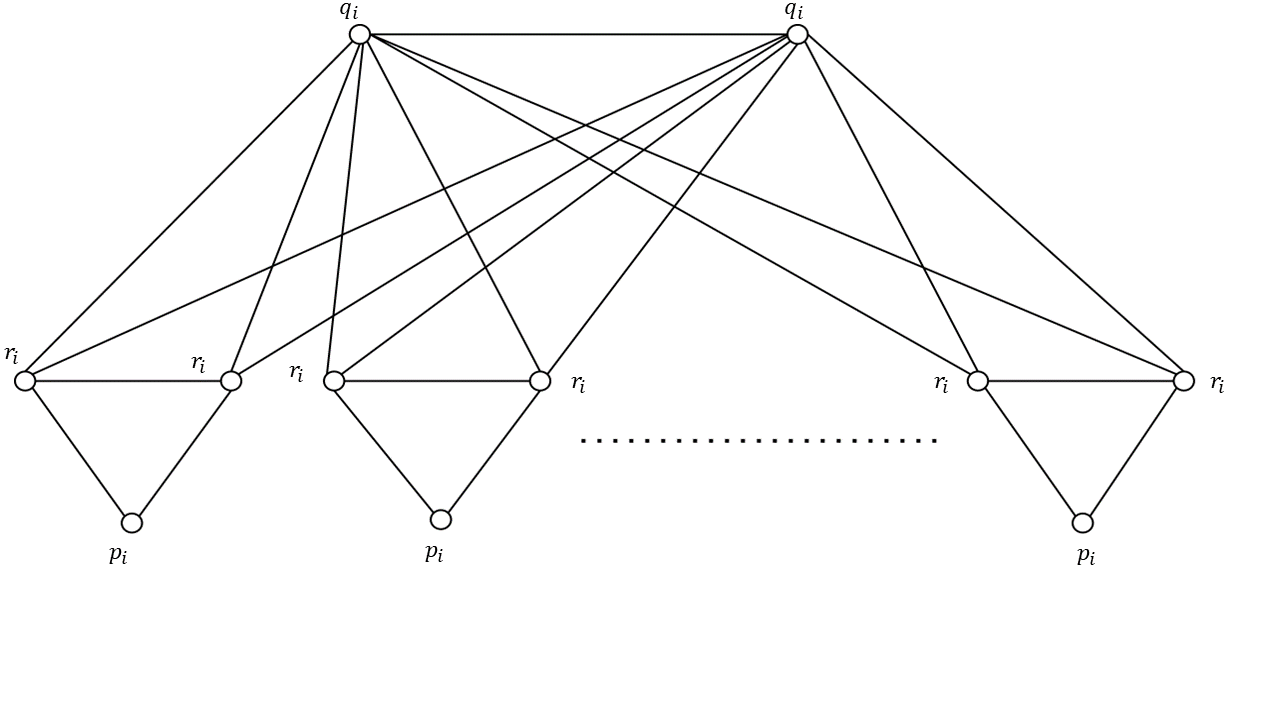}
    \caption{Proper Power Graph of $C_3 \oplus C_2 \oplus C_2 \oplus...\oplus C_2$}
    \label{fig-1}
\end{figure}

\textbf{case(ii)} Here we take $p$=3 and $q$=2, so the group $G$ is like $C_2 \oplus C_3 \oplus C_3 \oplus...\oplus C_3$, in this group one two-order element, $p^k$-1 three-order elements and remaining are six-order elements. So in $P^\ast(G)$, the vertex corresponding to the two-order element is adjacent to all vertices corresponding to the six-order element. Every vertex corresponding to the three-order elements is adjacent to one vertex corresponding to the three-order element and two vertices corresponding to the six-order element. Here every vertex corresponding to the three-order element has a degree three and every vertex corresponding to the six-order element has a degree four. Therefore, in this graph, only one vertex corresponding to the two-order element is adjacent to more than four vertices, and the remaining vertices corresponding to the six-order element are adjacent to one vertex corresponding to the six-order element and one vertex corresponding to the two-order element. Here $K_5$ is not possible as well as $k_{3,3}$ also, therefore, by Kuratowski’s theorem $P^\ast(G)$ is planar. In the below graph, $q_i$ is a vertex corresponding to the two-order element, $p_i$ is the vertices corresponding to the three-order element and $r_i$ is the vertices corresponding to the six-order element.  
\begin{figure}[htb!]
    \centering
    \includegraphics[width=100mm]{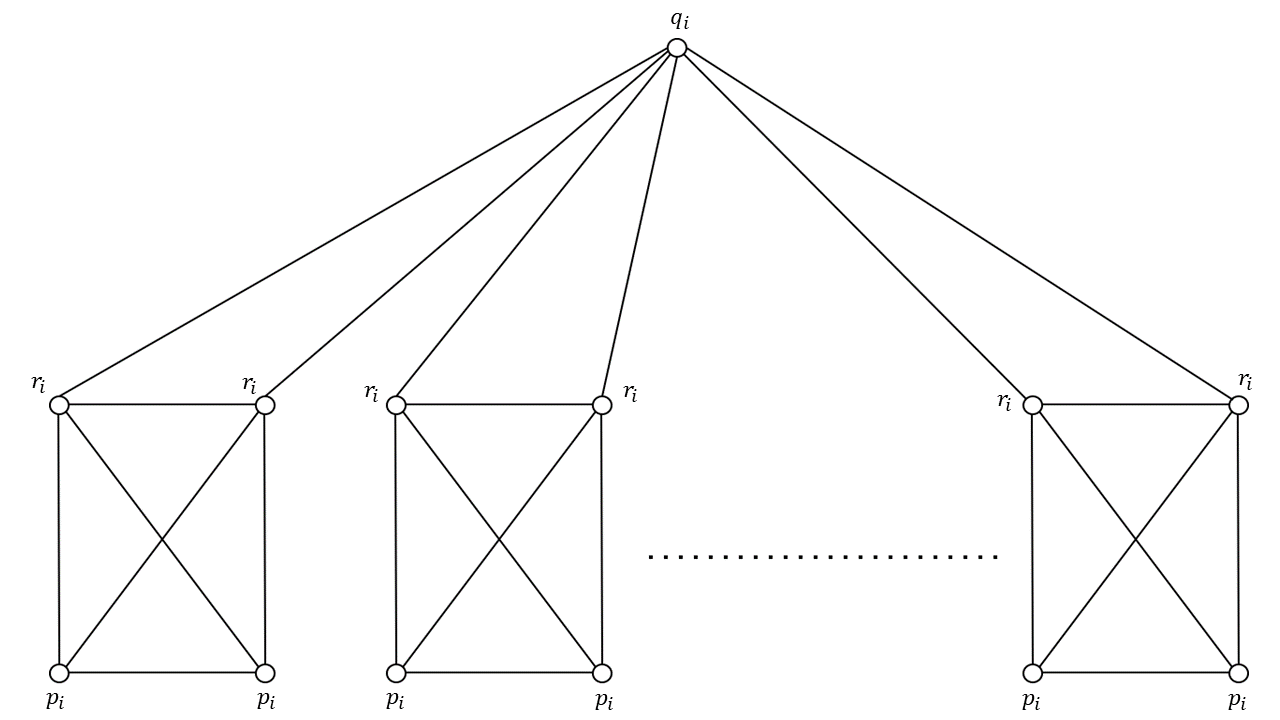}
    \caption{Proper Power Graph of $C_2 \oplus C_3 \oplus C_3 \oplus...\oplus C_3$}
    \label{fig-2}
\end{figure}
\end{proof}

\renewcommand{\baselinestretch}{0.1}

\bibliographystyle{ams}
\bibliography{bibtex}  
\end{titlepage}
\end{document}